\documentclass[11pt]{amsart}

\usepackage[letterpaper,margin=1in]{geometry}

\usepackage[T1]{fontenc}
\usepackage{lmodern}
\usepackage{microtype}
\usepackage{amsmath,amssymb,amsthm,mathtools}
\usepackage{mathrsfs}
\usepackage{enumitem}
\usepackage{booktabs,tabularx,array}
\usepackage[colorlinks=true,linkcolor=blue,citecolor=blue,urlcolor=blue]{hyperref}
\usepackage{aliascnt}
\usepackage[nameinlink,noabbrev]{cleveref}
\hypersetup{
  pdftitle={A simple C*-algebra which is not K1-injective},
  pdfauthor={Andrew S. Toms},
  pdfsubject={Higher-rank matrix trapping, spin bordism, and simple AH algebras}
}

\newtheorem{theorem}{Theorem}[section]
\newaliascnt{proposition}{theorem}
\newtheorem{proposition}[proposition]{Proposition}
\aliascntresetthe{proposition}
\newaliascnt{lemma}{theorem}
\newtheorem{lemma}[lemma]{Lemma}
\aliascntresetthe{lemma}
\newaliascnt{corollary}{theorem}
\newtheorem{corollary}[corollary]{Corollary}
\aliascntresetthe{corollary}
\theoremstyle{definition}
\newaliascnt{definition}{theorem}
\newtheorem{definition}[definition]{Definition}
\aliascntresetthe{definition}
\newaliascnt{construction}{theorem}
\newtheorem{construction}[construction]{Construction}
\aliascntresetthe{construction}
\theoremstyle{remark}
\newaliascnt{remark}{theorem}
\newtheorem{remark}[remark]{Remark}
\aliascntresetthe{remark}

\crefname{theorem}{theorem}{theorems}
\crefname{proposition}{proposition}{propositions}
\crefname{lemma}{lemma}{lemmas}
\crefname{corollary}{corollary}{corollaries}
\crefname{definition}{definition}{definitions}
\crefname{construction}{construction}{constructions}
\crefname{remark}{remark}{remarks}

\newcommand{\C}{\mathbb C}
\newcommand{\R}{\mathbb R}
\newcommand{\Z}{\mathbb Z}

\newcommand{\U}{\mathcal U}
\newcommand{\Uzero}{\mathcal U_0}
\newcommand{\End}{\operatorname{End}}
\newcommand{\Hom}{\operatorname{Hom}}
\newcommand{\rank}{\operatorname{rank}}
\newcommand{\Gr}{\operatorname{Gr}}
\newcommand{\pr}{\operatorname{pr}}
\newcommand{\id}{\mathrm{id}}

\newcommand{\Spin}{\mathrm{Spin}}
\newcommand{\red}{\operatorname{red}}

\newcommand{\thetaC}{\underline{\C}}

\newcommand{\wt}{\widetilde}

\newcommand{\KO}{\mathrm{KO}}
\newcommand{\KKO}{\mathrm{KKO}}
\newcommand{\pt}{\mathrm{pt}}

\title{A simple C$^*$-algebra which is not $K_1$-injective}
\author{Andrew S. Toms}
\date{Draft of \today}

\begin{document}

\begin{abstract}
We exhibit a simple unital infinite-dimensional C$^*$-algebra $A$ which is not $K_1$-injective, i.e., the canonical map
\[
 \U(A)/\Uzero(A)\longrightarrow K_1(A)
\]
is not injective.  This phenomenon is well known in the non-simple case.  Indeed, a natural example is the generator $u$ of $\pi_4(U(2))\cong\Z/2$, and we use exactly this example as the seed of our construction.  The construction proceeds by developing a higher-rank matrix trapping technique in homogeneous algebras generalizing the machinery of Villadsen in higher stable rank simple C$^*$-algebras.  A hypothetical homotopy from the homomorphic image of $u$
to the identity in a certain homogeneous algebra is shown, through matrix trapping, to yield a spin bordism from the class of $u$ to that of the identity matrix $\mathrm{1}_2 \in M_2$, and this, in turn, is obstructed by the fact that the bordism class of $u$ is not trivial.  We then fit our new arguments into a generalized version of Villadsen's higher stable rank construction to arrive at a simple example.

\end{abstract}

\maketitle

\section{Introduction}

For a unital $C^*$-algebra $A$ there is a canonical homomorphism
\begin{equation}\label{eq:canonical-K1-map-intro}
 \kappa_A:\U(A)/\Uzero(A)\longrightarrow K_1(A),
\end{equation}
where $\Uzero(A)$ denotes the path component of the identity $1_A \in \U(A)$.  We say that $A$
is \emph{$K_1$-injective} if $\kappa_A$ is injective, so that a unitary
which becomes null-homotopic after passage to a matrix algebra is already
null-homotopic in $\U(A)$.  The question of whether every simple C$^*$-algebra is $K_1$-injective dates to at least 1998 in the second edition of Blackadar's book \cite{Blackadar98}.  There, the question is actually whether $\kappa_A$ is {\it bijective}, but the question of surjectivity was answered negatively by Villadsen in 2002 \cite{Villadsen02}.   

Several familiar regularity hypotheses entail $K_1$-injectivity.  Stable rank one implies it via Rieffel's stable-rank theory \cite{Rieffel83}, \cite{Rieffel87}.  Jiang proved that
$\mathcal Z$-stable algebras are $K$-stable, and hence $K_1$-injective, in
\cite{Jiang97} (see also Hua \cite{Hua24}).  Other classical cases are purely infinite simple algebras
(Cuntz, \cite{Cuntz81}) and algebras of real rank zero
(Lin, \cite[Lemma~2.2]{Lin96}).  Villadsen showed, in a different direction,
that simple AH algebras can retain striking nonstable phenomena, including non-minimal stable
rank and failure of $K_1$-surjectivity \cite{Villadsen99,Villadsen02}.  Whether simplicity
alone forces $K_1$-injectivity has nevertheless remained open and appears as Problem LIX in
\cite{STW99Problems}.  The purpose of this paper is to give a negative answer to that question.

\begin{theorem}\label{thm:main-intro}
There exists a simple, separable, unital, infinite-dimensional, nuclear AH algebra $A$ such
that
\[
 K_1(A)=0
 \qquad\text{and}\qquad
 \U(A)/\Uzero(A)
\]
contains a nonzero element of order two.  In particular, $A$ is not $K_1$-injective.
\end{theorem}
\noindent
We would like to emphasize that while this result is negative in character, the matrix-trapping techniques we develop here should have further broad applications for questions which ask generally ``What phenomena in commutative or homogeneous C$^*$-algebras can persist in a simple C$^*$-algebra?''

Our starting point is the nonzero element of
$\pi_4(U(2))\cong\pi_4(SU(2))\cong\Z/2$, represented by a smooth map
$u:S^4\to SU(2)$.  This map is not null-homotopic in $U(2)$, but
$u\oplus1_2$ is null-homotopic in $U(4)$, so $M_2(C(S^4))$ already fails to be
$K_1$-injective.  The difficulty is to preserve this failure while adding the
matrix summands needed for simplicity.

Our construction follows the broad outlines of Villadsen's higher stable rank examples to arrive at a simple algebra
\cite{Villadsen98,Villadsen99}, but two further ingredients are needed here.
Villadsen's work on higher stable rank AH algebras introduces the idea of trapping a rank one corner of a homogeneous algebra with spectrum $M \times X$, where $M$ and $X$ are carefully chosen manifolds.  He shows that the corner is generally left stable over a certain locus $N$ for which the standard co-ordinate projection $\pi_M:N \to M$ induces an injective map on cohomology---$N$ sees all of $M$.  Our first significant innovation here is to generalize the corner trapping idea to higher rank corners.  This fits into a recent pattern of research in operator algebras emphasizing matrix geometry.  Examples include the author's construction of a simple nuclear C$^*$-algebra without uniform property $\Gamma$ and the proof of Elliott, Li, and Niu that a simple separable C$^*$-algebra need not be singly generated \cite{Toms26Gamma}, \cite{ElliottLiNiu26}.  Higher-rank corner trapping forces a $2\times2$ corner to decouple on a
geometric locus.  A second ingredient, spin bordism, supplies a homotopy-invariant
class which distinguishes the chosen unitary from the identity
even when the trapping locus changes topology.  The sequel uses bordism at a level which is presumably fairly elementary for an expert but represents an unfamiliar field for most operator algebraists.  We have therefore written the sequel with more detail than such an expert might in order to make it accessible to its target audience. 

We now give a coarse outline of the proof.  Any undefined or unfamiliar terms will be addressed carefully in later sections.  We mean only to give the flavor here.  Ultimately, our algebra $A$ is AH, so that $A \cong \lim (A_i,\phi_i)$ for homogeneous algebras $A_i$ and unital $*$-homomorphisms $\phi_i:A_i \to A_{i+1}$.  At a single stage, let $X$ be a closed, simply connected spin manifold of
dimension $4r$, and let $Q\to X$ be a Hermitian bundle of complex rank $r$.
Put
\[
 B=S^4\times X,\qquad E=\thetaC^2\oplus Q,\qquad
 W=\Hom(Q,\thetaC^2)\cong Q^*\oplus Q^*,
\]
with bundles pulled back from $X$.  The reader should think of the homogeneous algebra $\Gamma(\End(E))$ of sections the endomorphism bundle of $E$ as the first stage algebra $A_1$.  The upper-right block of a unitary section
$v$ is a section of $W$.  At a zero of this section the other off-diagonal
block also vanishes.  A nonzero Euler class of $W_{\R}$ therefore forces the
block-diagonal locus $Z(v)$---in plain terms, the set of points where $v$ is block diagonal---to meet every $X$-slice.

The normal bundle of the block-diagonal subbundle is the pullback
of $W_{\R}$, which is spin by a characteristic class calculation.  After a transverse
perturbation, $Z(v)$ is a closed four-manifold with the spin structure
induced by the ambient and normal spin structures.  The trapped $2 \times 2$ block
defines a map $a_v:Z(v)\to U(2)$.  Compose it with a fixed retraction
$\rho:U(2)\to SU(2)$ to obtain $f_v=\rho\circ a_v$, and let $c_v$ be the
constant map to the identity.  We define
\begin{equation}\label{eq:spin-class-intro}
 \nu_{X,Q}(v):=[Z(v),f_v]-[Z(v),c_v]
 \in\wt\Omega_4^{\Spin}(SU(2))\cong\Z/2.
\end{equation}
This is a whole lot of technology obscuring a fairly simple idea:
there is a homology theory which assigns a class in $\Z/2$ to
each of our unitaries.  For the choices of $X$ and $Q$ made
below, this class is nonzero on
$w=(u\circ\pr_{S^4})\oplus1_Q$ and zero on $1_E$.
Two cycles represent the same class when their maps extend
over a compatible spin bordism.  In the sequel we will prove
that $\nu_{X,Q}$ is homotopy invariant via such bordisms.
The differing classes of $w$ and $1_E$ then obstruct a
homotopy between them in the homogeneous building block.

The characteristic-number calculation in
\cref{prop:spin-class-formula} multiplies the nonzero class
of the seed by $\langle c_r(Q)^2,[X]\rangle$.  We arrange
this number to be one.  Products of suitable Grassmannians
allow us to preserve it, together with the spin and dimension
conditions, as the inductive system is enlarged.  The
point-evaluation summands have nontrivial tautological
multiplicity bundles.  Their constant unitary actions are
homotopic to the identity, but these homotopies do not
trivialize the bundles.

Choosing the evaluation points densely gives simplicity.
The spectra have only even-dimensional cells, so $K_1(A_i)=0$
at every stage, while the distinguished unitary is not in
the identity component.  If it became null-homotopic in the
limit, it would be null-homotopic at a finite stage.
This contradiction proves the theorem.

\Cref{sec:spin-invariant,sec:finite-stage} contain the construction,
homotopy invariance, and computation of the trapped class.  Readers may
initially accept the calculation for the suspended Hopf class in
\cref{lem:spin-detects-hopf} and consult \cref{sec:preliminaries} as needed.
\Cref{sec:grassmannians,sec:inductive-system} construct the twisted AH
system, and \cref{sec:main-proof} passes the obstruction to its simple
limit.  The real $K$-homology interpretation in
\cref{subsec:KO-interpretation} is not used in the proof.
Throughout, $K$ denotes complex $K$-theory and $KO$ denotes real
$K$-theory.  All algebras in the construction are complex.

 \noindent
 {\bf AI statement.}  All mathematical arguments and proofs are due to the author, and the author takes full responsibility for them.  ChatGPT 5.6 {\it was} used to search for references.
Several of these, which the author might not have found on
his own, clarified that our original real $K$-theory invariant
(which still works, incidentally) could be understood in terms
of spin bordism.  We then realized that the spin bordism viewpoint was enough to prove our main result from existing technology---the real K-theory approach just represented an inessential layer of complexity. 
 
\noindent
{\bf Acknowledgements.}  The author was generously supported by the Simons Foundation via SFI-MPS-TSM-00025606.  The author thanks Stuart White and Shanshan Hua for valuable discussions on this topic during the author's visit to Oxford in 2024, which encouraged him to revisit this problem in light of emerging applications of matrix geometry in C$^*$-algebras.

\section[Preliminaries: unstable K1 and spin bordism]
{Preliminaries: unstable \texorpdfstring{$K_1$}{K1} and spin bordism}
\label{sec:preliminaries}

A bordism joins two closed manifolds by realizing their disjoint
union as the boundary of a compact manifold of one higher dimension.
If the two manifolds carry maps into the same space, a bordism of
maps also requires these maps to extend to a single map on the
connecting manifold.  A homotopy is the special case in which that
manifold is a cylinder.  Bordism therefore allows us to compare
maps whose domains need not be the same manifold.

We will require these manifolds to carry spin structures, which
are additional data on their oriented tangent bundles.  A spin
bordism must induce the specified spin structures on its boundary,
so it imposes a condition beyond the existence of an oriented
bordism.  This extra condition will distinguish the unitary used
in our construction from the identity.  We begin with a description
of spin structures in local frames, before giving the bundle
formulation and the bordism calculation needed below.

\subsection*{A first look at spin structures}

Let $M$ be a smooth oriented $n$-manifold, and choose smoothly
varying inner products on its tangent spaces.  An orthonormal
frame at $x$ is an ordered orthonormal basis of $T_xM$.  On small
open sets $U_i$ we can choose such frames with the given orientation.
On an overlap $U_i\cap U_j$, the change of frame is a rotation
$g_{ij}(x)\in SO(n)$.  These transition functions satisfy
$g_{ij}g_{jk}=g_{ik}$ wherever three sets overlap.

The spin group comes with a two-to-one covering homomorphism
\[
 \lambda:\Spin(n)\longrightarrow SO(n),
 \qquad \ker\lambda=\{1,-1\}.
\]
In dimension two, both groups are circles and this is the map
$z\mapsto z^2$.  Each rotation has two lifts.  On a cover with
contractible overlaps, each $g_{ij}$ can be lifted continuously,
but arbitrary lifts need not satisfy the original compatibility
relations.  A \emph{spin structure} is described by lifts for which
\[
 \lambda(\widetilde g_{ij})=g_{ij},
 \qquad
 \widetilde g_{ij}\widetilde g_{jk}=\widetilde g_{ik}.
\]
The question is thus whether the signs of the lifts can be chosen
consistently on all overlaps.

A manifold equipped with such a choice for its tangent bundle is
a \emph{spin manifold}.  The definition applies equally to any
oriented real vector bundle.  A spin structure is not a global
frame: it adds lifting data without trivializing the bundle.
Existence and choice are separate issues, since an oriented bundle
may admit no spin structure or more than one up to isomorphism
\cite{Milnor63,LawsonMichelsohn89}.

The oriented circle already illustrates why the choice matters.
It has two spin structures, only one of which is induced from the
disk.  The other cannot occur as the sole boundary of a compact
spin surface, as we prove in \cref{lem:spin-circle}.  Thus a manifold
which bounds after we forget the spin structure need not bound
with its specified spin structure.

\subsection{The first unstable unitary class}

For a unital $C^*$-algebra $A$, write
\[
 \U_n(A)=\U(M_n(A)),
 \qquad
 \U_n^0(A)=\text{the identity component of }\U_n(A).
\]
The inclusions $v\mapsto v\oplus1$ define
\[
 K_1(A)=\varinjlim_n\U_n(A)/\U_n^0(A),
\]
and the map \eqref{eq:canonical-K1-map-intro} is induced by the first stabilization.

\begin{definition}\label{def:K1-injective}
A unital $C^*$-algebra $A$ is \emph{$K_1$-injective} if the canonical map
\[
 \kappa_A:\U(A)/\Uzero(A)\longrightarrow K_1(A)
\]
is injective.
\end{definition}

\begin{lemma}\label{lem:unstable-seed}
There is a smooth map
\[
 u:S^4\longrightarrow SU(2)\subseteq U(2)
\]
with the following properties:
\begin{enumerate}[label=\textup{(\roman*)}]
\item $u$ is not null-homotopic in $U(2)$.
\item $u\oplus1_2:S^4\to U(4)$ is null-homotopic.
\item the class of $u$ has order two in $[S^4,U(2)]$.
\end{enumerate}
Consequently $M_2(C(S^4))$ is not $K_1$-injective.
\end{lemma}

\begin{proof}
The determinant fibration
\[
 SU(2)\longrightarrow U(2)\xrightarrow{\det}S^1
\]
gives
\[
 \pi_4(U(2))\cong\pi_4(SU(2))\cong\pi_4(S^3)\cong\Z/2.
\]
Choose $u$ to represent the nonzero element, and under $SU(2)\cong S^3$ it may be taken to be
the suspension of the Hopf map \cite[Section~5]{PutmanNotes}.  This gives (i) and (iii).  Bott periodicity gives
$\pi_4(U)=0$, and $U(4)$ is in the stable range in degree four, so $u\oplus1_2$ is
null-homotopic.  Since
\[
 K_1(C(S^4))\cong K^{-1}(S^4)=0,
\]
the nonzero component of $u$ lies in the kernel of the canonical map for
$M_2(C(S^4))$.
\end{proof}

\subsection{Spin structures and induced boundary structures}

For an oriented real rank-$n$ bundle $V\to M$, choose a
Euclidean metric and write $P_{SO}(V)\to M$ for its
principal bundle of positively oriented orthonormal frames.

A \emph{spin structure} on $V$ is a principal $\Spin(n)$-bundle $P_{\Spin}(V)\to M$
together with an equivariant two-to-one map
\[
 P_{\Spin}(V)\longrightarrow P_{SO}(V)
\]
covering the identity of $M$ and equivariant for the double covering
$\lambda:\Spin(n)\to SO(n)$.  Thus a spin structure is a lift of the
oriented frame bundle through $\lambda$, not a global frame of $V$.
It need not trivialize the bundle.

A lift exists precisely when $w_2(V)=0$, and its isomorphism classes then
form a torsor for $H^1(M;\Z/2)$.  In particular, an oriented spin bundle
over a simply connected base has a unique spin structure up to
isomorphism.  These statements do not depend on the metric.  A
\emph{spin manifold} is a smooth oriented manifold with a specified spin
structure on its tangent bundle.  Products carry product spin structures.
On $S^4$ we use the unique spin structure, induced on the boundary of
the standard spin five-ball.

A complex bundle $F$ has a canonical real orientation.  Its underlying
real bundle need not be spin, since
\begin{equation}\label{eq:w2-complex-bundle}
 w_2(F_{\R})\equiv c_1(F)\pmod2.
\end{equation}
The obstruction vanishes for $F\oplus F$.  A spin structure on the base
does not imply that every oriented bundle over it is spin, since the tangent
bundle and each auxiliary bundle have their own lifting problems.
For these facts and compatibility with direct sums, see
\cite{LawsonMichelsohn89,MilnorStasheff74}.

For a submanifold, spin structures on the ambient tangent bundle and the
normal bundle determine a spin structure on its tangent bundle through a
tangent-normal splitting.  The following two-out-of-three rule specifies
that structure, rather than merely asserting its existence.

\begin{lemma}\label{lem:spin-two-out-of-three}
Let $M$ be an oriented spin manifold and let $Z\subseteq M$ be an embedded submanifold with
oriented normal bundle $\nu$.  Give $Z$ the orientation determined by the tangent-normal
splitting.  If $\nu$ has a spin structure and
\[
 TZ\oplus\nu\cong TM|_Z,
\]
then the spin structures on $TM|_Z$ and $\nu$ determine a spin structure on $TZ$.

If $M=[0,1]\times B$ has the product spin structure and $Z\subseteq M$ has product collars
\[
 Z\cap([0,\varepsilon)\times B)=[0,\varepsilon)\times Z_0,
 \qquad
 Z\cap((1-\varepsilon,1]\times B)=(1-\varepsilon,1]\times Z_1,
\]
then the induced boundary spin structures are exactly the two-out-of-three structures on
$Z_0$ and $Z_1$.  With the outward-normal-first convention,
\[
 \partial Z=Z_1\sqcup(-Z_0).
\]
\end{lemma}

\begin{proof}
After compatible metrics are chosen, the stable splitting identifies the oriented frame
bundle of $TM|_Z$ with the block sum of the frame bundles of $TZ$ and $\nu$.  The block-sum
homomorphism for spin groups lifts the corresponding homomorphism for special orthogonal
groups.  Lifts for two of the three bundles therefore determine the lift for the third.

On a product collar,
\[
 TM\cong\underline{\R}\oplus TB,
 \qquad
 TZ\cong\underline{\R}\oplus TZ_j,
\]
where the collar line carries its canonical spin structure.  Cancelling this common spin
line gives precisely the structure obtained from
$TZ_j\oplus\nu|_{Z_j}\cong TB|_{Z_j}$.  The boundary orientation formula is standard.
\end{proof}

Compatible metrics and splittings of the tangent-normal exact sequence
form contractible spaces.  Once the ambient and normal spin structures
and the normal-bundle identification are fixed, the induced spin
structure is therefore independent of these choices up to isomorphism.
We use this structure on each trapping locus.

\subsection{Spin bordism of mapped manifolds}

Fix a space $T$.  A spin-bordism cycle of dimension $n$ over $T$ is a triple
\[
 (M,\sigma,f),
\]
where $M$ is a closed smooth $n$-manifold, $\sigma$ is a spin structure on
$TM$, and $f:M\to T$ is continuous.  Closed manifolds are compact and
without boundary, but they may be disconnected or empty.  We usually
suppress $\sigma$, although it remains part of the cycle.

\begin{definition}\label{def:spin-bordism}
Two cycles $(M_0,\sigma_0,f_0)$ and $(M_1,\sigma_1,f_1)$ are \emph{spin bordant over
$T$} if there are a compact spin $(n+1)$-manifold $\mathcal W$ and a continuous map
$F:\mathcal W\to T$, together with boundary identifications
\[
 \partial\mathcal W=M_1\sqcup(-M_0),
\]
such that the induced boundary spin structures agree with the prescribed structures
and $F|_{M_i}=f_i$.  The notation $-M_0$ includes the opposite spin orientation.
\end{definition}

We use the outward-normal-first convention.  The outward normal
trivializes the boundary normal line.  Restricting the spin structure
on $T\mathcal W$ and cancelling this framed line gives the boundary
spin structure.  At the initial boundary the collar points inward,
accounting for the minus sign.  Compatibility with the product collars
used below is given by \cref{lem:spin-two-out-of-three}.

The resulting classes form an abelian group
\[
 \Omega_n^{\Spin}(T).
\]
Addition is disjoint union, the empty manifold represents zero, and reversal of spin
orientation gives the inverse.  The cylinder $[0,1]\times M$, with the map constant in
the interval direction, proves $[M,f]+[-M,f]=0$.  A null class is represented by the
boundary of a compact spin manifold over which the map to $T$ extends.

A spin filling of $M$ need not admit an extension of $f$, so $M$ may
bound while $(M,f)$ does not bound over $T$.  This occurs for the cycle
used below: $S^4$ bounds a spin five-ball, but the map $u:S^4\to S^3$
prevents the mapped cycle from bounding.  We call the homology theory
$\Omega_*^{\Spin}$ \emph{spin bordism} \cite{Stong68}.  Its groups of mapped
manifolds should not be confused with the Lie groups $\Spin(n)$ used
to define their tangent structures.

\subsection{Reduction and the one-dimensional coefficient group}

A map $g:T\to T'$ induces a homomorphism by composition:
\[
 g_*:\Omega_n^{\Spin}(T)\longrightarrow\Omega_n^{\Spin}(T'),
 \qquad [M,f]\longmapsto[M,g\circ f].
\]
Suppose $T$ has a basepoint $t_0$.  The collapse $p:T\to\pt$ and basepoint inclusion
$j:\pt\to T$ satisfy $p_*j_*=\id$.  The \emph{reduced} group is
\begin{equation}\label{eq:reduced-spin-group}
 \wt\Omega_n^{\Spin}(T)
 =\ker\bigl(p_*:\Omega_n^{\Spin}(T)\to\Omega_n^{\Spin}(\pt)\bigr).
\end{equation}
If $c_M:M\to T$ is constant at $t_0$, define
\begin{equation}\label{eq:reduced-spin-cycle}
 [M,f]_{\red}:=[M,f]-[M,c_M].
\end{equation}
Both terms use the same spin structure.  Their difference removes the
contribution of the underlying spin manifold, leaving the information
supplied by the map.  In particular, a constant map has zero reduced
class.  The difference in \eqref{eq:reduced-spin-cycle} is represented
by $(M,f)\sqcup(-M,c_M)$.

The only coefficient group needed here is the one-dimensional group.
An oriented circle has two spin structures.  The \emph{bounding}
structure is induced from the disk, while the \emph{nonbounding} structure
is induced by the positively oriented tangent framing of the circle
as a Lie group.  The underlying oriented circle is the same in both
cases.

\begin{lemma}\label{lem:spin-circle}
The group $\Omega_1^{\Spin}(\pt)$ is isomorphic to $\Z/2$, generated by the circle with
its nonbounding spin structure.
\end{lemma}

\begin{proof}
Every closed spin one-manifold is a disjoint union of spin circles.  Bounding circles
bound disks.  Reversing orientation preserves each of the two spin types, so the
nonbounding circle is isomorphic to its inverse and has order at most two.

It does not bound a spin surface.  Indeed, let $F$ be a connected oriented surface with
one boundary circle and cap it by a disk to obtain $\widehat F$.  The closed oriented
surface $\widehat F$ admits a spin structure.  Its restriction to $F$ induces the
bounding structure on the boundary.  Every other spin structure on $F$ differs by an
element of $H^1(F;\Z/2)$, whose restriction to the boundary is zero because that circle
is null-homologous in $F$.  Hence every spin structure on $F$ induces the bounding
boundary structure.  This rules out a spin filling of the other circle and proves the
claim (compare \cite[Theorem~2.1]{KirbyTaylor90}).
\end{proof}

Let $(M,\sigma,f)$ be a mapped spin four-manifold over $S^3$.
After smoothing $f$, choose a regular value $y$ different from the
basepoint and set $L=f^{-1}(y)$.  Then $L$ is a closed one-manifold,
hence a finite disjoint union of circles.  An oriented basis of
$T_yS^3$ and the derivative of $f$ give a framing of the rank-three
normal bundle of $L$ in $M$.  This framing and $\sigma$ induce a
spin structure on each component of $L$.

By \cref{lem:spin-circle}, the class of $L$ in
$\Omega_1^{\Spin}(\pt)$ is determined by the parity of its
nonbounding components.  This class is unchanged under spin
bordism of maps.  Indeed, after making such a bordism transverse
to $y$ relative to its boundary, the inverse image of $y$ is a
compact spin surface with the two endpoint inverse images as
boundary, carrying their induced spin structures.  A map constant
at the basepoint has empty inverse image and therefore contributes
zero.

Spin bordism is a generalized homology theory, so its suspension
isomorphism gives
\begin{equation}\label{eq:spin-sphere-suspension}
 \wt\Omega_n^{\Spin}(S^k)\cong\Omega_{n-k}^{\Spin}(\pt).
\end{equation}
In particular,
\begin{equation}\label{eq:spin-target-group}
 \wt\Omega_4^{\Spin}(S^3)
 \cong\Omega_1^{\Spin}(\pt)\cong\Z/2.
\end{equation}
Under the Pontryagin--Thom description, the isomorphism
\eqref{eq:spin-target-group} is given by
\begin{equation}\label{eq:regular-fiber-spin}
 [M,f]_{\red}\longmapsto[f^{-1}(y)]
 \quad\text{in }\Omega_1^{\Spin}(\pt)
\end{equation}
with the spin structure on $f^{-1}(y)$ defined above
\cite{Stong68}.  Thus the parity of the nonbounding components
determines the reduced class.

For comparison, write $\Omega_*^{SO}$ for oriented bordism.
Every oriented circle bounds an oriented disk, so
\[
 \wt\Omega_4^{SO}(S^3)
 \cong\Omega_1^{SO}(\pt)=0.
\]
Thus forgetting the spin structure loses the distinction in
\eqref{eq:spin-target-group}.

\subsection{The suspended Hopf class in spin bordism}

Fix the identity as basepoint of $SU(2)\cong S^3$, and use the smooth map $u$ of
\cref{lem:unstable-seed}.  Equip $S^4$ with its unique spin structure and set
\begin{equation}\label{eq:spin-Hopf-class}
 \eta_{\Spin}:=[S^4,u]_{\red}
 \in\wt\Omega_4^{\Spin}(S^3).
\end{equation}
The constant-map term bounds the spin five-ball.  We retain the reduced
notation to agree with the convention for general trapping loci.
To show that $\eta_{\Spin}$ is nonzero, it suffices to find an odd
number of nonbounding spin circles in a regular inverse image of $u$.
The following lemma does so using the framed-circle description of
the suspended Hopf map.

\begin{lemma}
\label{lem:spin-detects-hopf}
The class $\eta_{\Spin}$ is the nonzero element of
$\wt\Omega_4^{\Spin}(S^3)\cong\Z/2$.
\end{lemma}

\begin{proof}
Under Pontryagin--Thom, a regular inverse image of the suspended Hopf map represents the
stable Hopf element $\eta\in\pi_1^{\mathrm s}\cong\Z/2$ in framed bordism.  It can be
represented by a circle whose stable normal framing differs from a disk-bounding
framing by one full twist.  This is the usual framed-circle calculation of
$\pi_4(S^3)$ \cite[Section~5]{PutmanNotes}.

A full twist represents the nontrivial loop in the stable special orthogonal group.
Its lift to the spin double cover does not close.  Thus changing the normal framing by
this twist changes the induced spin structure on the circle from the bounding one to
the other one.  Forgetting the framing therefore leaves the nonbounding spin circle.
The spin structure used on the source $S^4$ is the one induced by its standard stable
framing, since $S^4$ has only one spin structure.

By \eqref{eq:regular-fiber-spin}, $\eta_{\Spin}$ maps to the class of that nonbounding
circle.  This class is nonzero by \cref{lem:spin-circle}.  The target in
\eqref{eq:spin-target-group} has order two, so $\eta_{\Spin}$ is its generator.
\end{proof}

\begin{remark}\label{rem:regular-value-hopf}
We use only the nonvanishing of $\eta_{\Spin}$.  The framed-circle
calculation proves this without computing any further
spin-bordism coefficient groups.
\end{remark}

\subsection{Euler numbers as signed zero counts}

Let $V\to M$ be an oriented real vector bundle of rank $k$ over a
closed oriented $k$-manifold.  A section $s$ transverse to the zero
section has isolated zeros, and its ordinary Euler number satisfies
\begin{equation}\label{eq:ordinary-Euler-zero-count}
 \sum_{x\in Z(s)}\operatorname{sign}_x(s)
 =\left\langle e(V),[M]\right\rangle.
\end{equation}
Here $\operatorname{sign}_x(s)$ is $+1$ or $-1$ according as
$D_xs:T_xM\to V_x$ preserves or reverses orientation.  For a complex bundle of rank
$r$, canonically oriented as a real bundle, $e(V_{\R})=c_r(V)$
\cite{MilnorStasheff74}.  In our application, each zero of $s$ gives a
copy of $S^4$ with the orientation specified by its sign.  Disjoint
union turns this signed count into a multiplier in spin bordism.
No generalized Euler class is needed.

\subsection{Smooth bundle models and relative transversality}

Let $B$ be a closed smooth manifold and let $E\to B$ be a Hermitian complex vector bundle.
We write $\Gamma(\End E)$ for the $C^*$-algebra of continuous endomorphism sections.  Its
unitary group is the space of continuous sections of the smooth bundle
\[
 \pi_{\U}:\U(E)\longrightarrow B
\]
whose fiber at $b$ is the unitary group of $E_b$.

Every continuous unitary section is homotopic to a smooth one.
Indeed, approximate a continuous unitary section $v$ uniformly by a
smooth section $x$ of $\End E$.  If $\|x-v\|<1$, then $x$ is fiberwise invertible and its polar part
\[
 \operatorname{pol}(x)=x(x^*x)^{-1/2}
\]
is a smooth unitary section.  For a sufficiently close approximation,
$\operatorname{pol}(x)$ is uniformly less than two from $v$ and can be
joined to it by the fiberwise logarithmic path.  The same argument
applies to homotopies over $[0,1]\times B$ and, by relative smooth
approximation, can be fixed on a closed subset where the original
section is already smooth \cite{Hirsch76}.

For the sections used below, transversality can be tested on one
off-diagonal block.  Suppose that $E=\thetaC^2\oplus Q$, where $Q$
has complex rank $r$, and write
\[
 v=
 \begin{pmatrix}
  a&b\\
  c&d
 \end{pmatrix}
\]
relative to this splitting.  The upper-right block $b$ is a section
of $\Hom(Q,\thetaC^2)$.  If $b(x)=0$, unitarity gives
\[
 a(x)a(x)^*=1_2,
 \qquad
 a(x)^*a(x)+c(x)^*c(x)=1_2.
\]
The square matrix $a(x)$ is therefore unitary, and the second equality
gives $c(x)=0$.  Thus the block-diagonal locus is exactly the zero set
of $b$.

The normal-space calculation in \cref{lem:normal-bundle-D} shows
that $v$ is transverse to the block-diagonal subbundle precisely
when, at every zero $x$ of $b$, the derivative
\[
 D_xb:T_xB\longrightarrow\Hom(Q_x,\C^2)_{\R}
\]
is surjective as a real linear map.  At a zero of $b$ this derivative
is intrinsic, although it may be computed in a local bundle frame.
The regular-value theorem then makes the zero set a smooth
submanifold of real codimension $4r$.

Perturbing the off-diagonal block alone need not preserve unitarity.
We must perturb within the unitary bundle and, for a homotopy
already transverse on endpoint collars, keep those collars fixed.
The following relative transversality result provides such
perturbations within the bundle fibers.

\begin{proposition}
\label{prop:relative-section-transversality}
Let $p:F\to M$ be a smooth fiber bundle over a compact smooth manifold, possibly with
boundary, and let $S\subseteq F$ be a closed smooth subbundle.  Let $C\subseteq M$ be
closed, and suppose that a smooth section $s:M\to F$ is transverse to $S$ on a
neighborhood of $C$.  Then every $C^\infty$-neighborhood of $s$ contains a smooth section
$s'$ such that
\[
 s'=s\quad\text{on a neighborhood of }C,
 \qquad
 s'\pitchfork S\quad\text{on all of }M.
\]
If $M$ has boundary and $C$ contains a boundary collar, the perturbation may be taken fixed
on that collar.
\end{proposition}

\begin{proof}
Choose open neighborhoods
\[
 C\subseteq U_0\subseteq\overline{U_0}\subseteq U_1
\]
such that $s$ is transverse to $S$ on $U_1$, and put $K=M\setminus U_0$.  A tangent vector to $F$ is called \emph{vertical} if it is tangent
to a fiber of $p:F\to M$.  Since $S$ is a subbundle, the vertical
tangent space at each point of $S$ surjects onto the normal space
of $S$ in $F$. Cover the compact set $s(K)$ by finitely many bundle charts.  Using
vertical coordinate fields, bump functions, and a partition of unity, choose smooth
vertical vector fields
\[
 X_1,\ldots,X_N
\]
which vanish over a neighborhood of $C$ and span the relevant vertical tangent spaces near
$s(K)$.

Let $\Phi_j^t$ be the local flow of $X_j$.  For a sufficiently small parameter
$\mathbf t=(t_1,\ldots,t_N)$, define
\[
 \mathcal F(\mathbf t,x)
 =\Phi_N^{t_N}\circ\cdots\circ\Phi_1^{t_1}(s(x)).
\]
Because all fields are vertical, $x\mapsto\mathcal F(\mathbf t,x)$ is a section and agrees
with $s$ near $C$.  At parameter zero the parameter derivatives span the vertical tangent
space over $K$, hence surject onto the normal space at every intersection with $S$ there,
while on $U_1$ transversality is already supplied by $Ds$.  Compactness and closedness of $S$
allow the parameter ball to be shrunk so that $\mathcal F$ is transverse to $S$ as a map
from the parameter ball times $M$.  Parametric transversality then provides arbitrarily
small regular parameters.  If a boundary collar is contained in $C$, every perturbing
field vanishes there.
\end{proof}

\section{The trapped spin-bordism class}\label{sec:spin-invariant}

Fix $r\geq1$.  Let $X$ be a closed, connected, simply connected
smooth spin manifold of real dimension $4r$, and let $Q\to X$
be a Hermitian complex vector bundle of rank $r$.
Put $B=S^4\times X$ and pull $Q$ back to $B$ without
changing notation.  Let
\begin{equation}\label{eq:E-splitting}
 E=\thetaC^2\oplus Q.
\end{equation}
The product manifold $B$ is spin.  Since it is simply connected, its spin structure is
unique.

\subsection{The block-diagonal subbundle}

Inside $\U(E)$, let
\[
 \mathcal D(E)\subseteq\U(E)
\]
be the smooth subbundle of block-diagonal unitaries.  The fiber over $(x,\xi)\in B$ is
\[
 U(2)\times U(Q_\xi)\subseteq U(\C^2\oplus Q_\xi).
\]
Set
\begin{equation}\label{eq:W-definition}
 W=\Hom(Q,\thetaC^2)\cong Q^*\oplus Q^*.
\end{equation}
It has complex rank $2r$ and real rank $4r$.

\begin{lemma}
\label{lem:normal-bundle-D}
There is a natural isomorphism of oriented real vector bundles
\[
 \nu\bigl(\mathcal D(E)\subseteq\U(E)\bigr)
 \cong
 \pi_{\mathcal D}^*W_{\R},
\]
where $\pi_{\mathcal D}:\mathcal D(E)\to B$ is the bundle projection.  In particular,
$\mathcal D(E)$ has real codimension $4r$ in $\U(E)$.
\end{lemma}

\begin{proof}
Fix $b\in B$ and write $E_b=E_{0,b}\oplus E_{1,b}$ with $E_{0,b}=\C^2$ and
$E_{1,b}=Q_b$.  Let
\[
 g=\begin{pmatrix}g_0&0\\0&g_1\end{pmatrix}
 \in U(E_{0,b})\times U(E_{1,b}).
\]
Left translation by $g^{-1}$ identifies $T_gU(E_b)$ with the skew-adjoint endomorphisms of
$E_b$.  Relative to the splitting, the Hilbert--Schmidt orthogonal complement of
$
 \mathfrak u(E_{0,b})\oplus\mathfrak u(E_{1,b})
$
consists of
\[
 X_z=
 \begin{pmatrix}
 0&z\\
 -z^*&0
 \end{pmatrix},
 \qquad z\in\Hom(E_{1,b},E_{0,b}).
\]
Thus $z\mapsto gX_z$ identifies the normal space at $g$ with $W_b$ as a real vector space.
If
\[
 h=\begin{pmatrix}h_0&0\\0&h_1\end{pmatrix}
\]
is a change of local unitary frame respecting the splitting, then
\[
 hX_zh^{-1}=X_{h_0zh_1^{-1}},
\]
which is exactly the transition law for $\Hom(Q,\thetaC^2)$.  The identifications assemble
to the asserted bundle isomorphism.  The transition maps are complex linear, so the
canonical real orientation is preserved.
\end{proof}

The bundle $W_{\R}$ is spin because
\begin{equation}\label{eq:w2-W-zero}
 w_2(W_{\R})\equiv c_1(W)=2c_1(Q^*)\equiv0\pmod2.
\end{equation}
Since $B$ is simply connected, this spin structure is unique.  Pull it
back to the normal bundle of $\mathcal D(E)$ using
\cref{lem:normal-bundle-D}.  We fix this structure on the whole
block-diagonal subbundle and transport it to each inverse-image locus
by the normal derivative of the unitary section.  The spin structures
on the individual loci are not chosen independently.

\subsection{Definition of the trapped class}

Let $v:B\to\U(E)$ be a smooth unitary section transverse to $\mathcal D(E)$.  Explicitly,
at every $b$ with $v(b)\in\mathcal D(E)$, the map
\[
 T_bB\xrightarrow{dv_b}T_{v(b)}\U(E)
 \longrightarrow\nu_{v(b)}\mathcal D(E)
\]
is surjective.  In a local block chart this says that the derivative of the upper-right
block is onto at its zeros.  The inverse image
\[
 Z(v):=v^{-1}(\mathcal D(E))
\]
is therefore a closed smooth four-manifold.  

Compose the derivative of $v$ with projection onto the normal bundle
of $\mathcal D(E)$ and the identification in
\cref{lem:normal-bundle-D}.  The resulting map vanishes on $TZ(v)$,
and transversality gives an isomorphism
\[
 \overline{dv}:
 \frac{TB|_{Z(v)}}{TZ(v)}
 \xrightarrow{\ \cong\ }W_{\R}|_{Z(v)}.
\]
The left-hand bundle is the normal bundle of $Z(v)$ in $B$.  Give it
the orientation and spin structure transported from $W_{\R}|_{Z(v)}$
through this derivative isomorphism.  This uses the normal spin structure
fixed on the block-diagonal subbundle.  Splitting the
tangent-normal exact sequence gives
\begin{equation}\label{eq:tangent-stable-isomorphism}
 TZ(v)\oplus W_{\R}|_{Z(v)}\cong TB|_{Z(v)}.
\end{equation}
The restricted ambient spin structure and the normal spin structure
determine a spin structure $\sigma_v$ on $TZ(v)$ by
\cref{lem:spin-two-out-of-three}.  Its isomorphism class does not
depend on the auxiliary splitting.

The manifold $Z(v)$ need not be simply connected.  Simple connectivity
of $B$ gives uniqueness of the initial spin structures on $TB$ and
$W_{\R}$, not uniqueness of all spin structures on $Z(v)$.  The
invariant uses the particular structure $\sigma_v$ induced by the
normal derivative and the splitting above.

On $Z(v)$ the unitary is block diagonal:
\[
 v|_{Z(v)}=
 \begin{pmatrix}
 a_v&0\\
 0&d_v
 \end{pmatrix}.
\]
Since the first summand of $E$ is globally trivial, its block defines
a smooth map $a_v:Z(v)\to U(2)$.  Fix the smooth retraction
\begin{equation}\label{eq:rho-retraction}
 \rho:U(2)\longrightarrow SU(2),
 \qquad
 \rho(A)=
 \begin{pmatrix}
 \det(A)^{-1}&0\\0&1
 \end{pmatrix}A.
\end{equation}
It fixes $SU(2)$ pointwise.  Put $f_v=\rho\circ a_v$, and let $c_v:Z(v)\to SU(2)$ be
the constant map at the identity.

\begin{definition}\label{def:spin-trapping-invariant}
For a transverse smooth unitary section $v$, define
\begin{equation}\label{eq:spin-trapping-invariant}
 \nu_{X,Q}(v)
 :=[Z(v),\sigma_v,f_v]-[Z(v),\sigma_v,c_v]
 =[Z(v),f_v]_{\red}
 \in\wt\Omega_4^{\Spin}(SU(2)).
\end{equation}
For an arbitrary continuous unitary section, choose a homotopic smooth transverse
representative and use this formula.  The next theorem proves independence of that
choice.
\end{definition}

By \eqref{eq:spin-target-group}, the target is
\[
 \wt\Omega_4^{\Spin}(SU(2))\cong\Omega_1^{\Spin}(\pt)\cong\Z/2.
\]
The constant-map subtraction makes a constant protected block contribute
zero, even when the underlying spin four-manifold has nonzero absolute
bordism class.

The two sections compared in \cref{sec:finite-stage} already require
transverse replacement:
\[
 w=(u\circ\pr_{S^4})\oplus1_Q
 \qquad\text{and}\qquad
 1_E.
\]
Both take values in $\mathcal D(E)$, so their normal derivatives vanish
and neither is transverse.  In \cref{prop:spin-class-formula} we
construct transverse representatives in their respective homotopy
classes.  Each replacement stays in its own class, and no homotopy
between $w$ and $1_E$ is assumed.

\begin{theorem}\label{thm:trapping-invariant-well-defined}
The class $\nu_{X,Q}(v)$ is independent of the chosen transverse representative and depends
only on the path component of $v$ in
\[
 \U\bigl(\Gamma(\End E)\bigr).
\]
Consequently \eqref{eq:spin-trapping-invariant} defines a function
\[
 \nu_{X,Q}:
 \U\bigl(\Gamma(\End E)\bigr)/\Uzero\bigl(\Gamma(\End E)\bigr)
 \longrightarrow
 \wt\Omega_4^{\Spin}(SU(2)).
\]
\end{theorem}

\begin{proof}
Let $w_0$ and $w_1$ be homotopic continuous unitary sections.  Choose
arbitrary smooth transverse representatives $v_j\sim_h w_j$ for
$j=0,1$.  They exist by smooth approximation and
\cref{prop:relative-section-transversality}, since sufficiently small
perturbations remain in the same unitary path component.
Concatenating the replacement paths with the given homotopy gives
\[
 v_0\sim_h w_0\sim_h w_1\sim_h v_1.
\]
It suffices to show that the mapped trapping loci of $v_0$ and $v_1$
define the same reduced spin-bordism class.  Taking $w_0=w_1$
then also proves independence of the representative chosen in
\cref{def:spin-trapping-invariant}.

Reparametrize the concatenated homotopy so that it is constant
near each endpoint.  Relative smooth approximation, fixed on
smaller endpoint collars, then gives a smooth unitary homotopy
\[
 V:[0,1]\times B\longrightarrow\U(\pr_B^*E).
\]
On the collars, $V$ equals $v_0$ or $v_1$ and is already transverse to
$[0,1]\times\mathcal D(E)$.  Apply
\cref{prop:relative-section-transversality} with these collars fixed.
An arbitrarily small perturbation through unitary sections gives
\[
 V\pitchfork[0,1]\times\mathcal D(E).
\]

\smallskip
\noindent
\emph{The smooth bordism.}
The inverse image
\[
 \mathcal Z:=V^{-1}([0,1]\times\mathcal D(E))
\]
is a compact smooth five-manifold.  Because the perturbation
was fixed on endpoint collars, for sufficiently small
$\varepsilon>0$ we have
\[
 \begin{aligned}
  \mathcal Z\cap([0,\varepsilon)\times B)
   &=[0,\varepsilon)\times Z(v_0),\\
  \mathcal Z\cap((1-\varepsilon,1]\times B)
   &=(1-\varepsilon,1]\times Z(v_1).
 \end{aligned}
\]
The two endpoint trapping loci are therefore the boundary components
of $\mathcal Z$.

\smallskip
\noindent
\emph{Compatibility of the spin structures.}
By \cref{lem:normal-bundle-D}, the normal derivative of $V$
identifies the normal bundle of $\mathcal Z$ with
$(\pr_B^*W_{\R})|_{\mathcal Z}$.  Transport the fixed spin structure
through this identification.  The resulting normal spin structure
and the product spin structure on $[0,1]\times B$ determine a
spin structure $\sigma_V$ on $\mathcal Z$ through the splitting
\[
 T\mathcal Z\oplus(\pr_B^*W_{\R})|_{\mathcal Z}
 \cong T([0,1]\times B)|_{\mathcal Z}
\]

Choose the splitting to be a product on the endpoint collars.
On the $j$th collar, $V$ is constant in the interval direction
and equals $v_j$, so its normal derivative is the pullback of
that of $v_j$.  It follows that the normal identification and
transported spin structure are the pullbacks of those used to
define $\sigma_{v_j}$.  At the endpoint, the splitting is
\[
 \underline{\R}\oplus TZ(v_j)\oplus W_{\R}|_{Z(v_j)}
 \cong
 \underline{\R}\oplus TB|_{Z(v_j)}.
\]
Cancelling the common collar line with its canonical spin
structure recovers \eqref{eq:tangent-stable-isomorphism}, with
the ambient and normal spin structures used to define
$\sigma_{v_j}$.  By \cref{lem:spin-two-out-of-three}, the
induced boundary spin structures are therefore the prescribed
endpoint structures.  With the outward-normal-first convention,
\[
 \partial\mathcal Z=Z(v_1)\sqcup(-Z(v_0))
\]
as spin manifolds, where the minus sign includes reversal of
the spin orientation.

\smallskip
\noindent
\emph{Extension of the protected map.}
The section $V$ is block diagonal on all of $\mathcal Z$.  Its
protected summand is globally trivial, so the upper-left block
defines a smooth map
\[
 a_V:\mathcal Z\longrightarrow U(2).
\]
Composing with the fixed retraction $\rho$ from
\eqref{eq:rho-retraction} gives
\[
 F:=\rho\circ a_V:\mathcal Z\longrightarrow SU(2),
 \qquad
 F|_{Z(v_j)}=f_{v_j}.
\]
Thus $(\mathcal Z,\sigma_V,F)$ is a spin bordism over $SU(2)$
between the endpoint cycles with their prescribed spin structures
and maps.  By \cref{def:spin-bordism},
\[
 [Z(v_0),\sigma_{v_0},f_{v_0}]
 =
 [Z(v_1),\sigma_{v_1},f_{v_1}]
\]
in $\Omega_4^{\Spin}(SU(2))$.

\smallskip
\noindent
\emph{Equality of the reduced classes.}
The constant map from $(\mathcal Z,\sigma_V)$ to the identity of
$SU(2)$ restricts to the maps $c_{v_0}$ and $c_{v_1}$ of
\cref{def:spin-trapping-invariant}.  Hence
\[
 [Z(v_0),\sigma_{v_0},c_{v_0}]
 =
 [Z(v_1),\sigma_{v_1},c_{v_1}]
\]
in $\Omega_4^{\Spin}(SU(2))$.  Subtracting this equality
from the preceding one yields
\[
 \nu_{X,Q}(v_0)=\nu_{X,Q}(v_1)
 \quad\text{in }\wt\Omega_4^{\Spin}(SU(2)).
\]
Since the transverse representatives were arbitrary, this proves both
homotopy invariance and independence of the choices in
\cref{def:spin-trapping-invariant}.
\end{proof}

\begin{remark}\label{rem:what-class-measures}
The invariant $\nu_{X,Q}$ associates a mapped spin-bordism class
to an unstable unitary component, but it is not asserted to be a
homomorphism.  The proof requires transversality of the total
section $V$ over $[0,1]\times B$, not of each section $V_t$.
Although $\mathcal Z$ is a smooth five-manifold, the projection
\[
 \mathcal Z\longrightarrow[0,1],
 \qquad (t,b)\longmapsto t,
\]
may have critical points, and a slice need not be a smooth
four-manifold.

For a local example, take $t,x\in\R$ and set
\[
 s_t(x)=x^2-t,
 \qquad
 s(t,x)=s_t(x).
\]
Since $\partial s/\partial t=-1$, the total map is transverse to
zero and
\[
 s^{-1}(0)=\{(x^2,x):x\in\R\}
\]
is a smooth parabola.  But $s_0$ is not transverse to zero: its
$x$-derivative vanishes at the origin.  There are no zeros for
$t<0$ and two for $t>0$.  Thus zeros can appear or disappear in
the slices of a smooth total zero set.

In the proof, the fixed transverse endpoint collars identify the
endpoint loci with the boundary of $\mathcal Z$.  The intermediate
loci may change topology and need not form a fiber bundle.  No
zero or component is followed through the homotopy.  The spin
structure and the protected block map on the entire $\mathcal Z$
give the bordism of endpoint cycles.
\end{remark}

\section{The finite-stage obstruction}\label{sec:finite-stage}

We compute the trapped class using explicit perturbations of the
identity and the distinguished unitary.  First consider
$X=\Gr_2(\C^4)$, the Grassmannian of complex two-planes in $\C^4$,
with its tautological rank-two bundle $Q=\zeta$.  A surjection
$L:\C^4\to\C^2$ defines a section
\[
 s(V)=L|_V
 \quad\text{of}\quad
 \Hom(\zeta,\thetaC^2).
\]
Its only zero is $V_0=\ker L$.  This zero is transverse and positive,
as verified in \cref{subsec:first-grassmannian-stage}.

Over $B=S^4\times X$, put $E=\thetaC^2\oplus\zeta$ and
$w=(u\circ\pr_{S^4})\oplus1_\zeta$.  Applying the unitary rotation
below to a small multiple of $s$ gives transverse representatives
of $1_E$ and $w$, both with trapping locus
\[
 S^4\times\{V_0\}.
\]
Their protected blocks on this copy of $S^4$ are $1_2$ and $u$,
respectively.  The first has zero reduced class, while the second
contributes $\eta_{\Spin}$.

Return now to $X$ and $Q$ as in \cref{sec:spin-invariant}.  A
transverse section may have several zeros, each giving a signed
copy of $S^4$.  The same calculation applies, with the Euler
number counting these copies.

\subsection{A unitary rotation}

The following standard rotation realizes a prescribed bundle map as the off-diagonal block of a unitary.

\begin{lemma}\label{lem:unitary-rotation}
Let $z:Q\to\thetaC^2$ be a smooth bundle map over $B$ with $\|z\|<1$.  Put
\begin{equation}\label{eq:Rz}
 R_z=
 \begin{pmatrix}
 (1-zz^*)^{1/2}&z\\
 -z^*&(1-z^*z)^{1/2}
 \end{pmatrix}
 \in\U\bigl(\Gamma(\End(\thetaC^2\oplus Q))\bigr).
\end{equation}
Then:
\begin{enumerate}[label=\textup{(\roman*)}]
\item $R_z$ is homotopic to the identity through the path $t\mapsto R_{tz}$.
\item $R_z$ is block diagonal exactly on $Z(z)$.
\item if $z$ is transverse to the zero section of $W$, then $R_z$ is transverse to $\mathcal D(E)$ and the induced normal identification agrees with $Dz$ up to the canonical identification in \cref{lem:normal-bundle-D}.
\end{enumerate}
If $g=g_0\oplus g_1$ is any smooth block-diagonal unitary section, then the upper-right block of $gR_z$ is $g_0z$.  Hence $gR_z$ has the same trapping locus as $R_z$, with the same transversality property.
\end{lemma}

\begin{proof}
Functional calculus gives
\[
 (1-zz^*)^{1/2}z=z(1-z^*z)^{1/2}.
\]
A direct block-matrix multiplication now shows that $R_z^*R_z=R_zR_z^*=1$.  The norm hypothesis is preserved under multiplication by $t\in[0,1]$, proving (i).  The upper-right block is exactly $z$.  For a unitary matrix, the vanishing of one off-diagonal block forces the other to vanish, so (ii) follows.  At a zero of $z$, the differential of the upper-right block is $Dz$, which is precisely the normal coordinate from \cref{lem:normal-bundle-D}, proving (iii).  The final assertion follows by block multiplication.  At a zero of $z$, left
translation by $g^{-1}$, as used in \cref{lem:normal-bundle-D}, removes the factor
$g_0$ from the upper-right normal derivative of $gR_z$.  The derivative of $g$ itself
is tangent to the block-diagonal subbundle.  Thus the same normal identification is
used for the two transverse loci.
\end{proof}

\subsection{The characteristic-number formula}

The top Chern class of $W=Q^*\oplus Q^*$ is
\begin{equation}\label{eq:top-chern-W}
 c_{2r}(W)=c_r(Q^*)^2=c_r(Q)^2.
\end{equation}
The sign from dualizing disappears on squaring.  The ordinary Euler
number is therefore
\begin{equation}\label{eq:Euler-computation}
 \left\langle e(W_{\R}),[X]\right\rangle
 =\left\langle c_{2r}(W),[X]\right\rangle
 =\left\langle c_r(Q)^2,[X]\right\rangle.
\end{equation}

\begin{proposition}\label{prop:spin-class-formula}
Let $X,Q,B,E$ be as in \cref{sec:spin-invariant}, let
$u:S^4\to SU(2)$ represent the nonzero element of $\pi_4(SU(2))$, and put
\[
 w=(u\circ\pr_{S^4})\oplus1_Q.
\]
Then
\begin{equation}\label{eq:spin-class-formula}
 \nu_{X,Q}(1_E)=0,
 \qquad
 \nu_{X,Q}(w)
 =\left\langle e(W_{\R}),[X]\right\rangle\,\eta_{\Spin}
 =\left\langle c_r(Q)^2,[X]\right\rangle\eta_{\Spin}.
\end{equation}
\end{proposition}

\begin{proof}
Choose a smooth section
\[
 s\in\Gamma(X,W)
\]
transverse to the zero section.  Since
$\dim_{\R}X=\rank_{\R}W$, its zero set is finite.
Pull $s$ back to $B=S^4\times X$, choose $\varepsilon>0$
so that $\|\varepsilon s\|<1$, and set
\[
 R=R_{\varepsilon s}.
\]
By \cref{lem:unitary-rotation}, the paths
\[
 t\longmapsto R_{t\varepsilon s},
 \qquad
 t\longmapsto wR_{t\varepsilon s}
\]
join $1_E$ to $R$ and $w$ to $wR$, respectively.
Both $R$ and $wR$ are transverse to $\mathcal D(E)$.
Their relevant features are:
\begin{center}
\begin{tabular}{@{}lcc@{}}
\toprule
 & $R$ & $wR$\\
\midrule
Homotopic to
 & $1_E$ & $w$\\
Upper-right block
 & $\varepsilon s$
 & $(u\circ\pr_{S^4})\varepsilon s$\\
Trapping locus
 & $S^4\times Z(s)$
 & $S^4\times Z(s)$\\
Protected block on each copy
 & $1_2$ & $u$\\
\bottomrule
\end{tabular}
\end{center}

To compare the induced spin structures on this common locus, fix
$(y,\xi)\in S^4\times Z(s)$.  In the $X$ directions, the
upper-right derivatives of $R$ and $wR$ are
\[
 \varepsilon D_\xi s
 \qquad\text{and}\qquad
 u(y)\varepsilon D_\xi s,
\]
respectively.  In the $S^4$ directions these derivatives
vanish, since $s(\xi)=0$.

At this point $R=1_E$ and $wR=u(y)\oplus1_{Q_\xi}$.
The normal coordinate of \cref{lem:normal-bundle-D}
uses left translation by the inverse of the corresponding
block-diagonal unitary.  For $wR$, it therefore removes
the factor $u(y)$:
\[
 u(y)^{-1}\bigl(u(y)\varepsilon D_\xi s\bigr)
 =\varepsilon D_\xi s.
\]
The normal identifications therefore agree, as in the proof of
\cref{lem:unitary-rotation}.  They transport the fixed spin structure
on $W_{\R}$ to the same normal spin structure.  Since the ambient
spin structure is also the same, the induced spin structures on the
common trapping locus agree.

For each $\xi\in Z(s)$, the normal derivative
\[
 \varepsilon D_\xi s:T_\xi X\longrightarrow W_\xi
\]
is an isomorphism with orientation sign $\operatorname{sign}_\xi(s)$,
since $\varepsilon>0$.  The tangent-normal convention gives
$S^4\times\{\xi\}$ the standard orientation at a positive zero
and the opposite orientation at a negative one.  For either
orientation, $S^4$ has a unique spin structure.  Each zero thus
contributes a signed copy of the same spin four-manifold.

It remains to compare the maps carried by these copies.
For $R$, the protected block is constantly $1_2$.
The mapped cycle and its constant-map counterpart in
\cref{def:spin-trapping-invariant} are therefore
identical, so
\[
 \nu_{X,Q}(1_E)=\nu_{X,Q}(R)=0.
\]
For $wR$, the protected block on every copy is $u$.
Since $\rho$ fixes $SU(2)$ pointwise, the map used in
the invariant is also $u$.  Additivity under disjoint
union of spin-bordism cycles gives
\[
 \begin{aligned}
  \nu_{X,Q}(w)
   &=\nu_{X,Q}(wR)\\
   &=\sum_{\xi\in Z(s)}
       \operatorname{sign}_\xi(s)\,[S^4,u]_{\red}.
 \end{aligned}
\]

Finally, the signed zero count is the ordinary Euler
number:
\[
 \sum_{\xi\in Z(s)}\operatorname{sign}_\xi(s)
 =\left\langle e(W_{\R}),[X]\right\rangle
 =\left\langle c_r(Q)^2,[X]\right\rangle,
\]
where the last equality is \eqref{eq:Euler-computation}.
Together with
$[S^4,u]_{\red}=\eta_{\Spin}$ from
\eqref{eq:spin-Hopf-class}, this proves
\eqref{eq:spin-class-formula}.
\end{proof}

\begin{theorem}\label{thm:finite-stage}
Let $X$ be a closed, connected, simply connected spin manifold of real dimension $4r$, and
let $Q\to X$ be a Hermitian complex vector bundle of rank $r$.  Suppose that
\begin{equation}\label{eq:odd-characteristic-number}
 \left\langle c_r(Q)^2,[X]\right\rangle\equiv1\pmod2.
\end{equation}
Set
\[
 B=S^4\times X,
 \qquad
 E=\thetaC^2\oplus Q,
 \qquad
 A=\Gamma(\End E),
\]
and let
\begin{equation}\label{eq:distinguished-stage-unitary}
 w=(u\circ\pr_{S^4})\oplus1_Q\in\U(A).
\end{equation}
Then
\[
 w\notin\Uzero(A),
 \qquad
 [w]_1=0\in K_1(A).
\]
In particular, $A$ is not $K_1$-injective.
\end{theorem}

\begin{proof}
By \cref{prop:spin-class-formula,lem:spin-detects-hopf} and
\eqref{eq:odd-characteristic-number},
\[
 \nu_{X,Q}(w)=\eta_{\Spin}\neq0,
 \qquad
 \nu_{X,Q}(1_E)=0.
\]
Homotopy invariance of the trapped class therefore gives $w\notin\Uzero(A)$.

For the stable complex $K_1$-class, reorder the summands of
\[
 E\oplus E
 =\thetaC^2\oplus Q\oplus\thetaC^2\oplus Q
\]
to obtain $\thetaC^4\oplus Q\oplus Q$.  Under this permutation,
$w\oplus1_E$ becomes
\[
 (u\oplus1_2)\oplus1_{Q\oplus Q}.
\]
By \cref{lem:unstable-seed}, $u\oplus1_2$ is null-homotopic in $U(4)$.  Extending the
homotopy by the identity on $Q\oplus Q$ shows that $w\oplus1_E$ is null-homotopic in
$\U(M_2(A))$.  Hence $[w]_1=0$ in $K_1(A)$.
\end{proof}

\begin{remark}
The hypothesis in \eqref{eq:odd-characteristic-number} says exactly that the ordinary Euler
number of the normal bundle is odd.  Formula \eqref{eq:spin-class-formula} then reduces this
integer modulo two because $\eta_{\Spin}$ has order two in spin bordism.
\end{remark}

\subsection{The first Grassmannian stage}
\label{subsec:first-grassmannian-stage}

For the example at the beginning of this section, put
$X=\Gr_2(\C^4)$ and let $\zeta\to X$ be its tautological rank-two
bundle.  A surjection $L:\C^4\to\C^2$
defines the section $s(V)=L|_V$ of
$\Hom(\zeta,\thetaC^2)\cong\zeta^*\oplus\zeta^*$.
It vanishes exactly at $V_0=\ker L$, where its derivative
\[
 \Hom(V_0,\C^4/V_0)\longrightarrow\Hom(V_0,\C^2)
\]
is the isomorphism induced by $L$.  Thus the zero is
transverse, the ordinary Euler number is one, and the
trapped class is exactly $\eta_{\Spin}$.

\begin{corollary}\label{cor:first-homogeneous-example}
Let $\zeta$ be the tautological rank-two bundle over $\Gr_2(\C^4)$ and set
\[
 A_0=\Gamma\!\left(\End(\thetaC^2\oplus\zeta)\right)
\]
over $S^4\times\Gr_2(\C^4)$.  Then the rank-two corner embedding
\[
 \iota:M_2(C(S^4))\longrightarrow A_0,
 \qquad
 \iota(f)=f\oplus0_\zeta,
\]
preserves the failure of $K_1$-injectivity: if $u$ is as in
\cref{lem:unstable-seed}, then
\[
 \iota(u)+1_{A_0}-\iota(1)=u\oplus1_\zeta
\]
is not null-homotopic but has zero $K_1$-class.
\end{corollary}

\begin{proof}
The Grassmannian is simply connected and spin, has real dimension eight, and
\[
 \left\langle c_2(\zeta)^2,[\Gr_2(\C^4)]\right\rangle=1.
\]
These facts are proved in \cref{lem:grassmannian-package}.  Apply
\cref{thm:finite-stage}.
\end{proof}

\begin{remark}\label{rem:minimal-rank-two}
The rank-one critical case cannot have odd Euler number.  If $X$ is a closed spin
four-manifold and $L\to X$ is a complex line bundle, the Wu formula \cite{MilnorStasheff74} gives
\[
 \left\langle c_1(L)^2,[X]\right\rangle
 \equiv
 \left\langle c_1(L)w_2(TX),[X]\right\rangle
 \equiv0\pmod2.
\]
Thus $r=2$ is the smallest possible rank in \cref{thm:finite-stage}, and
$\Gr_2(\C^4)$ realizes the corresponding smallest critical dimension.
\end{remark}

\section{The Grassmannian recursion}\label{sec:grassmannians}

To apply the finite-stage obstruction throughout an inductive system,
we must preserve the relation $\dim_{\R}X_i=4\rank(Q_i)$, which makes
the trapping loci four-dimensional.  We also need $X_i$ to be spin
and simply connected: together with \eqref{eq:w2-W-zero}, these
conditions give the ambient and normal spin structures used to define
the invariant.  The remaining requirement is an odd characteristic
number, so that \cref{prop:spin-class-formula} gives a nonzero class.

We will keep the characteristic number equal to one.  The spaces
will also have finite CW decompositions with only even-dimensional
cells, giving $K_1(A_i)=0$ by \cref{prop:K1-zero}.  The following
Grassmannians have all the required properties.

For $d\geq1$, put
\begin{equation}\label{eq:Gd-definition}
 G(d)=\Gr_2(\C^{2d+2}),
\end{equation}
and let $\zeta_d\to G(d)$ be the tautological rank-two bundle.

\begin{lemma}\label{lem:grassmannian-package}
For every $d\geq1$:
\begin{enumerate}[label=\textup{(\roman*)}]
\item $G(d)$ is a closed, connected, simply connected complex manifold of complex dimension $4d$.
\item $G(d)$ is spin.
\item
\begin{equation}\label{eq:grassmannian-number}
 \left\langle c_2(\zeta_d)^{2d},[G(d)]\right\rangle=1.
\end{equation}
\item $G(d)$ has a finite CW decomposition with only even-dimensional cells.
\end{enumerate}
\end{lemma}

\begin{proof}
The complex dimension of $\Gr_2(\C^n)$ is $2(n-2)$, which is $4d$ for $n=2d+2$.  The Schubert cell decomposition has one zero-cell and only even real-dimensional cells.  This proves connectedness, simple connectedness, and (iv).

Let $\eta_d$ denote the tautological quotient bundle of rank $2d$.  The tangent bundle is
\[
 TG(d)\cong\Hom(\zeta_d,\eta_d)
 \cong\zeta_d^*\otimes\eta_d.
\]
Since $\zeta_d\oplus\eta_d$ is trivial,
\[
 c_1(\eta_d)=c_1(\zeta_d^*).
\]
Consequently
\begin{align*}
 c_1(TG(d))
 &=2d\,c_1(\zeta_d^*)+2c_1(\eta_d)\\
 &=(2d+2)c_1(\zeta_d^*).
\end{align*}
This class is even.  For a complex manifold, $w_2(TG(d))$ is the mod-two reduction of $c_1(TG(d))$, and hence $G(d)$ is spin.

For (iii), use the standard Schubert convention
\[
 c_2(\zeta_d^*)=\sigma_{1,1}.
\]
Since $c_2(\zeta_d^*)=c_2(\zeta_d)$, the vertical Pieri rule gives
\[
 \sigma_{1,1}^{j}=\sigma_{j,j}
 \qquad(0\leq j\leq2d).
\]
The class $\sigma_{2d,2d}$ is the class of a point, with integral one.  This proves \eqref{eq:grassmannian-number}.  We refer to Fulton \cite{Fulton97} for these Schubert-calculus conventions and the Pieri rule.
\end{proof}

Suppose the bundle at stage $i$ has rank $N_i=r_i+2$, and the next
connecting map requires $m_i$ point evaluations.  Each acts on an
$N_i$-dimensional Hermitian space $V_{i,\ell}$.  Tensoring with a
rank-two tautological multiplicity bundle adds $2N_i$ to the
complementary rank, for a total increase of $2m_iN_i$.
To preserve the dimension relation, we must add $8m_iN_i$ real
dimensions to the base.  We therefore take
\[
 d_i=m_iN_i,
 \qquad G(d_i)=\Gr_2(\C^{2d_i+2}),
 \qquad \dim_{\R}G(d_i)=8m_iN_i.
\]
By \eqref{eq:grassmannian-number}, this factor contributes a
multiplier of one to the characteristic number.  The resulting
recursion is as follows.

\begin{proposition}\label{prop:recursive-bundles}
Suppose that $X_i$ is a closed, connected, simply connected spin manifold and that $Q_i\to X_i$ is a complex vector bundle of rank $r_i$ such that
\begin{equation}\label{eq:recursive-hypotheses}
 \dim_{\R}X_i=4r_i,
 \qquad
 \left\langle c_{r_i}(Q_i)^2,[X_i]\right\rangle=1.
\end{equation}
Put
\[
 N_i=r_i+2.
\]
Let $m_i\geq1$, let
\[
 V_{i,1},\ldots,V_{i,m_i}
\]
be Hermitian vector spaces of dimension $N_i$, set $d_i=m_iN_i$, and define
\begin{align}
 X_{i+1}&=X_i\times G(d_i),\label{eq:Xi-recursion}\\
 Q_{i+1}&=\pr_{X_i}^*Q_i
 \oplus
 \bigoplus_{\ell=1}^{m_i}
 \left(V_{i,\ell}\otimes\pr_{G(d_i)}^*\zeta_{d_i}\right).
 \label{eq:Qi-recursion}
\end{align}
Then $X_{i+1}$ is closed, connected, simply connected, and spin, and if
\[
 r_{i+1}=\rank(Q_{i+1}),
 \qquad
 N_{i+1}=r_{i+1}+2,
\]
then
\begin{align}
 r_{i+1}&=r_i+2m_iN_i,\label{eq:r-recursion}\\
 N_{i+1}&=(2m_i+1)N_i,\label{eq:N-recursion}\\
 \dim_{\R}X_{i+1}&=4r_{i+1},\label{eq:dimension-recursion}\\
 \left\langle c_{r_{i+1}}(Q_{i+1})^2,[X_{i+1}]\right\rangle&=1.
 \label{eq:chern-number-recursion}
\end{align}
Moreover, if $X_i$ has a finite CW decomposition with only even-dimensional cells, then so does $X_{i+1}$.
\end{proposition}

\begin{proof}
The first assertions follow from \cref{lem:grassmannian-package} and permanence of the stated properties under products.  Each bundle
\[
 V_{i,\ell}\otimes\zeta_{d_i}
\]
has rank $2N_i$, so \eqref{eq:r-recursion} follows.  Then
\[
 N_{i+1}=r_i+2+2m_iN_i=(2m_i+1)N_i,
\]
which proves \eqref{eq:N-recursion}.  By \cref{lem:grassmannian-package},
\[
 \dim_{\R}G(d_i)=8d_i=8m_iN_i,
\]
and therefore
\[
 \dim_{\R}X_{i+1}
 =4r_i+8m_iN_i
 =4r_{i+1}.
\]

After choosing a basis of $V_{i,\ell}$ solely for the purpose of calculating characteristic classes, the bundle $V_{i,\ell}\otimes\zeta_{d_i}$ is isomorphic to $\zeta_{d_i}^{\oplus N_i}$.  Hence its top Chern class is $c_2(\zeta_{d_i})^{N_i}$.  The top Chern class of a direct sum is the product of the top Chern classes of its summands, and therefore
\[
 c_{r_{i+1}}(Q_{i+1})
 =\pr_{X_i}^*c_{r_i}(Q_i)\,
  \pr_{G(d_i)}^*c_2(\zeta_{d_i})^{m_iN_i}.
\]
Squaring and evaluating on the product fundamental class gives
\begin{align*}
 \left\langle c_{r_{i+1}}(Q_{i+1})^2,[X_{i+1}]\right\rangle
 &=
 \left\langle c_{r_i}(Q_i)^2,[X_i]\right\rangle
 \left\langle c_2(\zeta_{d_i})^{2d_i},[G(d_i)]\right\rangle\\
 &=1
\end{align*}
by \eqref{eq:recursive-hypotheses} and \cref{lem:grassmannian-package}.  The final assertion follows by taking the product of the even-cell decompositions.
\end{proof}

\begin{construction}\label{con:recursive-data}
Start with
\[
 X_1=G(1)=\Gr_2(\C^4),
 \qquad
 Q_1=\zeta_1,
 \qquad
 r_1=2,
 \qquad
 N_1=4.
\]
At stage $i$, choose a positive integer $m_i$ and Hermitian vector spaces
\[
 V_{i,1},\ldots,V_{i,m_i}
\]
of dimension $N_i$, and apply \cref{prop:recursive-bundles}.  Then, for every $i$,
\begin{equation}\label{eq:all-recursive-properties}
\begin{gathered}
 \dim_{\R}X_i=4r_i,\qquad X_i\text{ is simply connected and spin},\\
 \left\langle c_{r_i}(Q_i)^2,[X_i]\right\rangle=1.
\end{gathered}
\end{equation}
Every $X_i$ has a finite CW decomposition with only even-dimensional cells.  In the inductive system below, the spaces $V_{i,\ell}$ will be the actual fibers of the preceding homogeneous bundle at the selected evaluation points, which makes the connecting maps canonical without choosing bases.
\end{construction}

\section{The Villadsen-type inductive system}\label{sec:inductive-system}

We construct a simple inductive limit in which the finite-stage obstruction
survives.  The coordinate summand preserves the distinguished unitary.
The twisted point-evaluation summands give simplicity, while their
Grassmannian multiplicity bundles keep the characteristic number equal
to one.

\subsection{Choice of evaluation points and intrinsic multiplicity bundles}

The spaces, bundles, and evaluation sets are chosen simultaneously.
Start with
\[
 X_1=G(1)=\Gr_2(\C^4),
 \qquad
 Q_1=\zeta_1,
 \qquad
 r_1=2,
 \qquad
 N_1=4,
\]
and put
\begin{equation}\label{eq:Yi-Ei-Ai}
 Y_i=S^4\times X_i,
 \qquad
 E_i=\thetaC^2\oplus Q_i,
 \qquad
 A_i=\Gamma(\End E_i)
\end{equation}
whenever the stage-$i$ data have been defined.  Equip each $Y_i$ with a compatible metric.  For $j\leq i$, let
\[
 \alpha_{j,i}:Y_i\longrightarrow Y_j
\]
be the coordinate projection, which preserves the $S^4$ coordinate and discards the Grassmannian factors added after stage $j$.

Suppose that the data have been constructed through stage $i$.  Choose a finite set
\begin{equation}\label{eq:Fi}
 F_i=\{z_{i,1},\ldots,z_{i,m_i}\}\subseteq Y_i
\end{equation}
so that, for every $1\leq j\leq i$, the set
\begin{equation}\label{eq:dense-projections}
 \alpha_{j,i}(F_i)
\end{equation}
is $1/i$-dense in $Y_j$.  To obtain such a set, choose a finite
$1/i$-net in each $Y_j$, lift its points to $Y_i$ using surjectivity
of the coordinate projection, and take the union of the lifts.

For each selected point, retain the actual Hermitian fiber
\begin{equation}\label{eq:Viell}
 V_{i,\ell}:=E_i|_{z_{i,\ell}}.
\end{equation}
This vector space has dimension $N_i$.  Set
\[
 d_i=m_iN_i,
 \qquad
 G_i=G(d_i),
 \qquad
 \zeta_i=\zeta_{d_i}\longrightarrow G_i,
\]
and define
\begin{align}
 X_{i+1}&=X_i\times G_i,\label{eq:Xi-system}\\
 Q_{i+1}&=\pr_{X_i}^*Q_i
 \oplus
 \bigoplus_{\ell=1}^{m_i}
 \left(V_{i,\ell}\otimes\pr_{G_i}^*\zeta_i\right).
 \label{eq:Qi-system}
\end{align}
By \cref{prop:recursive-bundles}, all the properties in \eqref{eq:all-recursive-properties} continue to hold, and each $X_i$ has only even-dimensional cells.

We have $Y_{i+1}=Y_i\times G_i$, and we write
$\alpha_i:Y_{i+1}\to Y_i$ and $\beta_i:Y_{i+1}\to G_i$
for the coordinate projections.  By the definition \eqref{eq:Qi-system}, there is a bundle decomposition
\begin{equation}\label{eq:E-next-intrinsic}
 E_{i+1}
 =
 \alpha_i^*E_i
 \oplus
 \bigoplus_{\ell=1}^{m_i}
 \left(V_{i,\ell}\otimes\beta_i^*\zeta_i\right).
\end{equation}
The $\ell$th point-evaluation block uses the actual fiber
$E_i|_{z_{i,\ell}}$, so no identification
$V_{i,\ell}\cong\C^{N_i}$ is needed.

\subsection{The connecting maps}

Using \eqref{eq:E-next-intrinsic}, define a unital $*$-homomorphism
\begin{equation}\label{eq:phi-i}
 \phi_i:A_i\longrightarrow A_{i+1}
\end{equation}
by
\begin{equation}\label{eq:phi-i-formula}
 \phi_i(a)
 =
 \alpha_i^*a
 \oplus
 \bigoplus_{\ell=1}^{m_i}
 \bigl(a(z_{i,\ell})\otimes1_{\beta_i^*\zeta_i}\bigr).
\end{equation}
The coordinate summand $\alpha_i^*a$ makes $\phi_i$ injective.

\begin{remark}\label{rem:not-ordinary-diagonal}
The point-evaluation blocks of the distinguished unitary are homotopic
to identity automorphisms, but their underlying bundles are not
thereby trivialized.  Let $C\in U(V_{i,\ell})$, and choose a path
$C_t$ from $C$ to $1_{V_{i,\ell}}$.  Tensoring gives a path
\[
 C_t\otimes1_{\beta_i^*\zeta_i}
\]
on the fixed bundle $V_{i,\ell}\otimes\beta_i^*\zeta_i$, ending at
its identity automorphism.  This is the path used in
\cref{lem:wi-propagates}.  In contrast, the null-homotopy of
$u\oplus1_2$ in \cref{lem:unstable-seed} takes place on a globally
trivial rank-four bundle.  It gives no null-homotopy of
$u\oplus1_{Q_i}$ merely from the fact that the second block is an
identity automorphism.

The point-evaluation multiplicity projection in
\eqref{eq:phi-i-formula} is the tautological projection defining
$\zeta_i$, not a constant projection.  The formula is diagonal
only in local bundle frames.  Our system is therefore twisted,
or generalized diagonal, rather than diagonal between trivial
full matrix bundles in the sense of \cite{ElliottHoToms09}.
The Elliott--Ho--Toms stable-rank-one theorem for simple unital
diagonal AH limits does not apply.  Replacing the tautological
multiplicity bundles by trivial ones would remove the Chern
number on which the trapped class depends.
\end{remark}

Let
\begin{equation}\label{eq:wi}
 w_i=(u\circ\pr_{S^4})\oplus1_{Q_i}\in\U(A_i).
\end{equation}
By \eqref{eq:all-recursive-properties} and \cref{thm:finite-stage},
\begin{equation}\label{eq:wi-nontrivial}
 w_i\notin\Uzero(A_i)
 \qquad(i\geq1).
\end{equation}

\begin{lemma}\label{lem:wi-propagates}
For every $i$,
\[
 \phi_i(w_i)\sim_h w_{i+1}
 \quad\text{in }\U(A_{i+1}).
\]
Consequently, for $j>i$,
\[
 \phi_{i,j}(w_i)\sim_h w_j.
\]
\end{lemma}

\begin{proof}
On the coordinate summand $\alpha_i^*E_i$, the unitary $\phi_i(w_i)$ is
\[
 (u\circ\pr_{S^4})\oplus1_{\alpha_i^*Q_i},
\]
which agrees with the corresponding part of $w_{i+1}$.  On the point-evaluation summand indexed by $\ell$, it is
\[
 w_i(z_{i,\ell})\otimes1_{\zeta_i}.
\]
The operator $w_i(z_{i,\ell})$ is a single unitary on the Hermitian vector space $V_{i,\ell}$.  The group $U(V_{i,\ell})$ is path connected, so choose a path from $w_i(z_{i,\ell})$ to $1_{V_{i,\ell}}$.  Tensoring that path with $1_{\zeta_i}$ gives a path of unitary bundle automorphisms from the point-evaluation block to the identity.  Taking the direct sum of these paths proves the first assertion.  The second follows by iteration.
\end{proof}

\subsection{Simplicity}

We use the familiar fullness criterion for diagonal AH systems.

\begin{lemma}\label{lem:fullness-criterion}
Let $(B_i,\psi_i)$ be an inductive system of unital $C^*$-algebras with injective connecting maps.  Suppose that for every $i$ and every nonzero positive $a\in B_i$, there is $j>i$ such that $\psi_{i,j}(a)$ is full in $B_j$.  Then $\varinjlim(B_i,\psi_i)$ is simple.
\end{lemma}

\begin{proof}
This is standard.  We recall the argument.  Let $I$ be a nonzero ideal in the limit $B$.  Choose $0\neq x\in I_+$.  Approximate $x$ closely by the image of a positive element $a\in B_i$.  After replacing $a$ by a nonzero cut-down $(a-\varepsilon)_+$, the usual perturbation lemma for positive elements shows that its image belongs to $I$.  By hypothesis, its image is full at a later stage and hence full in the limit.  Therefore $I=B$.
\end{proof}

\begin{proposition}\label{prop:limit-simple}
The inductive limit
\begin{equation}\label{eq:A-limit}
 A=\varinjlim(A_i,\phi_i)
\end{equation}
is simple.
\end{proposition}

\begin{proof}
Fix $i$ and let $0\neq a\in(A_i)_+$.  The set
\[
 U=\{y\in Y_i:a(y)\neq0\}
\]
is a nonempty open subset of $Y_i$.  By \eqref{eq:dense-projections}, for some $j\geq i$ there is a point $z\in F_j$ such that
\[
 \alpha_{i,j}(z)\in U.
\]
The coordinate block of $\phi_{i,j}(a)(z)$ is
\[
 a(\alpha_{i,j}(z)),
\]
so $\phi_{i,j}(a)(z)\neq0$.  In the next connecting map, the point-evaluation block
\[
 \phi_{i,j}(a)(z)\otimes1_{\zeta_j}
\]
occurs in $\phi_{i,j+1}(a)$ at every point of $Y_{j+1}$.  It is nonzero in every fiber.  Ideals of a homogeneous algebra $\Gamma(\End E)$ are in bijection with open subsets of the base, and the ideal generated by a section has support equal to the set on which its fiber value is nonzero.  Thus a section is full exactly when it is nonzero in every fiber.  It follows that $\phi_{i,j+1}(a)$ is full.  Apply \cref{lem:fullness-criterion}.
\end{proof}

\subsection[The K1-group of the limit]{The \texorpdfstring{$K_1$}{K1}-group of the limit}

\begin{proposition}\label{prop:K1-zero}
For every $i$,
\[
 K_1(A_i)=0.
\]
Consequently,
\[
 K_1(A)=0.
\]
\end{proposition}

\begin{proof}
The algebra $A_i=\Gamma(\End E_i)$ is strongly Morita equivalent to $C(Y_i)$, so
\[
 K_1(A_i)\cong K^1(Y_i).
\]
By \cref{con:recursive-data}, $X_i$ has a finite CW decomposition with only even-dimensional cells, and the same is true of $Y_i=S^4\times X_i$.  The Atiyah--Hirzebruch spectral sequence, or induction over the cells, gives
\[
 K^1(Y_i)=0.
\]
Continuity of operator $K$-theory for inductive limits now yields
\[
 K_1(A)=\varinjlim K_1(A_i)=0.
\]
See \cite{AtiyahKTheory,RordamLarsenLaustsen00} for the topological and operator-algebraic facts used here.
\end{proof}

\section{Passage of the obstruction to the limit}\label{sec:main-proof}

It remains to show that the distinguished unitary cannot become
null-homotopic in the completion of the algebraic limit.  For an
injective system, a finite-stage unitary which is null-homotopic in
the limit is already null-homotopic at some later finite stage.

\begin{lemma}\label{lem:path-finite-stage}
Let
\[
 B=\varinjlim(B_i,\psi_i)
\]
be an inductive limit of unital $C^*$-algebras with injective connecting maps.  Let $v\in\U(B_i)$.  If the image of $v$ in $B$ belongs to $\Uzero(B)$, then there is $j\geq i$ such that
\[
 \psi_{i,j}(v)\in\Uzero(B_j).
\]
\end{lemma}

\begin{proof}
Identify the $B_i$ with nested unital subalgebras of $B$,
and let $h:[0,1]\to\U(B)$ be a path from $v$ to $1$.
Choose a partition
\[
 0=t_0<t_1<\cdots<t_n=1
\]
such that $\|h(t_{k+1})-h(t_k)\|<1/2$ for every $k$.

Approximate the finitely many values $h(t_k)$ in a common
stage $B_j$, with $j\geq i$.  Take the endpoint
approximations to be $x_0=\psi_{i,j}(v)$ and $x_n=1$
exactly.  Choose all approximations sufficiently close
that each $x_k$ is invertible and its unitary polar part
\[
 y_k=x_k(x_k^*x_k)^{-1/2}\in\U(B_j)
\]
satisfies $\|y_k-h(t_k)\|<1/4$.  This is possible by
continuity of polar decomposition near unitaries.
The endpoints are unchanged, and
\[
 \|y_{k+1}-y_k\|<1
 \qquad(0\leq k<n).
\]
In particular, $-1$ is not in the spectrum of
$y_k^*y_{k+1}$.  The continuous logarithm gives a path
\[
 t\longmapsto
 y_k\exp\bigl(t\log(y_k^*y_{k+1})\bigr),
 \qquad 0\leq t\leq1,
\]
from $y_k$ to $y_{k+1}$ in $\U(B_j)$.
Concatenating these paths joins $\psi_{i,j}(v)$ to $1$.
\end{proof}

\begin{theorem}\label{thm:main}
The algebra $A$ in \eqref{eq:A-limit} is a simple, separable, unital, infinite-dimensional, nuclear AH algebra with
\[
 K_1(A)=0
\]
which is not $K_1$-injective.  More precisely, the image
\[
 w=\phi_{1,\infty}(w_1)\in\U(A)
\]
represents a nonzero element of order two in $\U(A)/\Uzero(A)$.
\end{theorem}

\begin{proof}
Each $A_i$ is a unital homogeneous algebra over a compact finite CW complex, and each connecting map is a unital injective $*$-homomorphism.  Thus $A$ is separable, unital, nuclear, and AH.  Since the injective image of $A_1$ is infinite-dimensional, so is $A$.  It is simple by \cref{prop:limit-simple}, and $K_1(A)=0$ by \cref{prop:K1-zero}.

Suppose, for contradiction, that
\[
 w=\phi_{1,\infty}(w_1)
\]
belongs to $\Uzero(A)$.  By \cref{lem:path-finite-stage}, there is $j\geq1$ such that
\[
 \phi_{1,j}(w_1)\in\Uzero(A_j).
\]
By \cref{lem:wi-propagates}, $\phi_{1,j}(w_1)$ is homotopic to $w_j$.  Hence $w_j\in\Uzero(A_j)$, contradicting \eqref{eq:wi-nontrivial}.  Therefore $w\notin\Uzero(A)$.  Since $K_1(A)=0$, its component lies in the kernel of
\[
 \U(A)/\Uzero(A)\longrightarrow K_1(A).
\]

Finally, the class of $u$ has order two in $[S^4,U(2)]$.  Thus $u^2$ is null-homotopic in $U(2)$, and
\[
 w_1^2=u^2\oplus1_{Q_1}
\]
is null-homotopic in $\U(A_1)$.  Its image $w^2$ is therefore null-homotopic in $A$.  The component of $w$ has order exactly two.
\end{proof}

\begin{proof}[Proof of \cref{thm:main-intro}]
This is \cref{thm:main}.
\end{proof}

\section{Further perspective}\label{sec:further-remarks}

\subsection{What the trapping construction retains}

The invariant preserves a mapped spin-bordism class rather than
the topology of the trapping locus.  Its spin structure is induced
from the ambient and normal bundles, and reduction removes the
constant-map contribution.  In \cref{prop:spin-class-formula}, the
class is the Euler number times $\eta_{\Spin}$.  The recursion
retains the spin hypotheses and the odd multiplier.

\subsection{A broader trapped-corner principle}

The argument suggests a more general use of higher-rank corner trapping.
Matrix data restrict to a degeneracy locus, and a suitable tangential
structure makes the mapped locus a bordism cycle.  A homotopy of
matrices then gives a bordism of these cycles.  To obtain an
obstruction, one must choose a theory in which the class carried by
the protected block is nonzero.

Here spin bordism both describes the geometry and detects the
class.  A natural transformation to another generalized homology
theory may give a further detector, but is not needed to construct
the trapped class.  This separation may be useful for other
obstructions.  It does not imply that every unstable class survives,
or that the resulting correspondence reflects positivity in a
bundle-order problem.

\subsection[An ancillary real K-homology interpretation]
{An ancillary real \texorpdfstring{$K$}{K}-homology interpretation}
\label{subsec:KO-interpretation}

The trapped class also has a real $K$-homology interpretation, which is
not needed for the proof.  Write $KO$ for real topological $K$-theory,
with homological grading.  For a compact space $T$, its
operator-algebraic realization is
\begin{equation}\label{eq:KO-homology-definition}
 \KO_n(T)=\KKO_n(C(T,\R),\R),
\end{equation}
where $\KKO$ denotes real Kasparov theory
\cite{Kasparov81,Blackadar98,Schick04}.  A closed spin $n$-manifold $M$ has a real,
Clifford-linear Dirac fundamental class $[D_M]\in\KO_n(M)$.  The Clifford structure
records the homological degree.  The Atiyah--Bott--Shapiro orientation
\cite{ABS64,ABP67} is the natural transformation
\begin{equation}\label{eq:ABS-transformation}
 \alpha_{\mathrm{ABS},T}:\Omega_n^{\Spin}(T)\longrightarrow\KO_n(T),
 \qquad [M,f]\longmapsto f_*[D_M].
\end{equation}
It sends spin bordisms of maps to equal $K$-homology classes.  On reduced classes,
\begin{equation}\label{eq:ancillary-KO-image}
 \alpha_{\mathrm{ABS},T}([M,f]_{\red})
 =f_*[D_M]-(c_M)_*[D_M].
\end{equation}
In \eqref{eq:KO-homology-definition}, the first pushforward is the Kasparov product
$[f^*]\otimes_{C(M,\R)}[D_M]$, and the constant-map term has the same form.

The coefficient group $\KO_1(\pt)=\KO_1(\R)$ is $\Z/2$, and the Dirac class of the
nonbounding spin circle maps to its nonzero element.  On the standard circle with
this spin structure, the spinors are periodic and the real mod-two Dirac index is one
\cite{AtiyahSinger69}.  Suspension and \cref{lem:spin-circle} therefore give an
isomorphism
\begin{equation}\label{eq:ABS-reduced-isomorphism}
 \alpha_{\mathrm{ABS},S^3}:\wt\Omega_4^{\Spin}(S^3)
 \xrightarrow{\ \cong\ }\wt\KO_4(S^3)
 \cong\KO_1(\pt)\cong\Z/2.
\end{equation}
In particular, $\eta_{\Spin}$ has nonzero image.  Applying
\eqref{eq:ancillary-KO-image} to $\nu_{X,Q}(v)$ therefore gives an
equivalent mod-two detector, without entering the construction of
the class or the proof of homotopy invariance.  The isomorphism in
\eqref{eq:ABS-reduced-isomorphism} concerns this reduced group and
degree only, and spin bordism and real $K$-homology are not being
identified in general.

The corresponding reduced complex $K$-homology group is zero:
\[
 \wt K_4(S^3)\cong K_1(\pt)=0.
\]
The real $K$-homology image therefore retains information lost in the
complex group.  It is a class associated to a mapped spin manifold,
not an assertion that $KO_1(A)$ contains a particular nonzero
element.  Our algebras are complex and have vanishing complex
$K_1$-groups.  The invariant used in the proof is $\nu_{X,Q}$,
and real $K$-homology gives another description of its obstruction.

\subsection{Regularity of the example}

The example is a simple unital nuclear AH algebra in the UCT class.
Its failure of $K_1$-injectivity rules out stable rank one and
$\mathcal Z$-stability by the results cited in the introduction.
We have not tried to optimize its trace simplex, stable rank, or
dimension growth.

\bibliographystyle{amsplain}
\bibliography{references}

\end{document}